\documentclass[11pt]{amsart}
\usepackage[utf8x]{inputenc}
\usepackage{graphicx, color}
\usepackage{amsmath,amssymb,amsthm}
\usepackage{hyperref}
\def\boxit{$\sqcap\kern-8pt\sqcup$}

\newcommand{\Z}{{\mathbb Z}}

\newtheorem{thm}{Theorem}[section]
\newtheorem{cor}[thm]{Corollary}

\newtheorem{lem}[thm]{Lemma}
\newtheorem{rem}[thm]{Remark}

\newtheorem{examp}[thm]{Example}

\newcommand{\R}{\mathbb R}

\newcommand{\bases}{{\mathcal B}}

\newcommand{\st}{\ensuremath{{\rm st}}}
\newcommand{\vol}{\ensuremath{{\rm vol}}}
\newcommand{\ideals}{\ensuremath{{\mathcal{I}}}}
\newcommand{\joinirrs}{\ensuremath{\mathcal{J}}}

\begin{document}
\title[On lattice path matroid polytopes]{On lattice path matroid polytopes: integer points and Ehrhart polynomial}
\thanks{The first author was partially supported by ANR grant GATO ANR-16-CE40-0009-01. The second author was supported by the Israel Science Foundation grant No. 1452/15 and the European Research Council H2020 programme grant No. 678765. The last two authors were partially supported by ECOS Nord project M13M01.}
\author[Knauer]{Kolja Knauer}
\address{Laboratoire d'Informatique Fondamentale, Aix-Marseille Universit\'e and CNRS,
Facult\'e des Sciences de Luminy, F-13288 Marseille Cedex 9, France}
\email{kolja.knauer@lif.univ-mrs.fr}
\author[Mart\'inez-Sandoval]{Leonardo Mart\'inez-Sandoval}
\address{Dept. of Computer Science, Faculty of Natural Sciences, Ben-Gurion University of the Negev, Beer Sheva, 84105, Israel}
\email{leomtz@im.unam.mx}
\author[Ram\'irez Alfons\'in]{Jorge Luis Ram\'irez Alfons\'in}
\address{Institut Montpelliérain Alexander Grothendieck, Universit\'e de Montpellier, Place Eug\`ene Bataillon, 34095 Montpellier Cedex, France}
\email{jorge.ramirez-alfonsin@umontpellier.fr}


\begin{abstract}
In this paper we investigate the number of integer points lying in dilations of lattice path matroid polytopes. We give a characterization of such points as polygonal paths in the diagram of the lattice path matroid. Furthermore, we prove that lattice path matroid polytopes
are affinely equivalent to a family of distributive polytopes. As applications we obtain two new infinite families of matroids verifying a conjecture 
of De Loera et.~al. and present an explicit formula of the Ehrhart polynomial for one of them.
\end{abstract}

\maketitle

\section{Introduction}

For general background on matroids we refer the reader to~\cite{Ox11,We76}. 
A {\em matroid}  $M=(E, \bases )$ of {\em rank} $r=r(M)$ is a finite set $E=\{1,\dots ,n\}$
together with a non-empty collection $\bases =\bases(M)$ of 
$r$-subsets of $E$ (called the {\em bases} of $M$) satisfying the following {\em basis exchange axiom}:
$$\text{if } B_1,B_2\in \bases \text{ and } e\in B_1\setminus B_2, \text{ then there exists }$$
$$ f\in B_2\setminus B_1 \text{ such that } (B_1-e)+ f \in \bases.$$
For a matroid $M=(E,\bases )$, the {\it matroid basis polytope} $P_M$ of $M$ is
defined as the convex hull of the incidence vectors of bases of $M$, that is,
$$
P_M:= \text{conv} \left\{\sum_{i\in B}e_{i} : B \text{~a base of~} M\right\},
$$
here $e_{i}$ is the $i^{th}$ standard basis vector in $\R^{n}$.  It is well-known~\cite{Fei-05} that $\dim(P_M)=n-c$ where $c$ is the number of {\em connected components} of $M$.  
\smallskip

Let $k\geq 0$ be an integer and let $P\subseteq \R^n$ be a polytope. We define $kP:=\{kp : p\in P\}$ and the function
$$L_P(k):=\#(kP\cap\Z^{n}).$$ Note that $L_P(0)=1$. It is well-known~\cite{Ehr} that for integral polytopes, including the case of matroid basis polytopes, the function $L_P$ extends to a polynomial on $\mathbb{R}$, called the {\em Ehrhart polynomial} of $P$. The {\em Ehrhart series} of a polytope $P$ is the infinite series 
$$Ehr_P(z)=\sum\limits_{k\ge 0}L_P(k)z^k.$$

{As a consequence  of the above polynomiality result we have} that if $P\subseteq \R^n$ is an integral convex polytope of dimension $n$ then its Ehrhart series is a rational function, 
$$Ehr_P(z)=\frac{h^*_P(z)}{(1-t)^{n+1}}=\frac{h^*_0+h^*_1t+\cdots +h^*_{n-1}t^{n-1}+h_n^*t^n}{(1-t)^{n+1}}.$$
 
The coefficients of $h^*_P$ are the entries of the {\em $h^*$-vector} of $P$. The function $L_P(t)$
can be expressed as

\begin{equation}\label{er-ser}
L_P(t)=\sum_{j=0}^nh^*_j{t+n-j\choose n}.
\end{equation}


In this paper, we first investigate the  function $L_{P_M}(t)$ when $M$ is a {\em lattice path matroid}.  
{The class of lattice path matroids was first introduced by Bonin, de Mier, and Noy~\cite{Bon-03}. Many different aspects of lattice path matroids have been studied:  excluded minor results~\cite{Bon-10}, algebraic geometry notions~\cite{Del-12,Sch-10,Sch-11}, the Tutte polynomial~\cite{Bon-07,KMR,Mor-13}, the associated basis polytope in connection with the combinatorics of Bergman complexes~\cite{Del-12}, its facial structure~\cite{An-17,Bid-12}, specific decompositions in relation with Lafforgue's work
~\cite{Cha-11} as well as the related cut-set expansion conjecture~\cite{Coh-15}.}
\smallskip

{In Section~\ref{sec:lpm}, we review some notions on lattice path matroids and introduce the interesting subclass of {\em snakes}. In Section~\ref{sec:basepol}, we provide a combinatorial characterization (in terms of some polygonal paths in the standard diagram representation of lattice path matroids $M$) of the points in $kP_M$ (Theorem~\ref{thm:polypaths}) as well as} of the integer points in $kP_M$ (Corollaries~\ref{cor:integ} and~\ref{cor:inter}). As an application, we obtain an explicit formula for $L_{P_M}(t)$ for an infinite family of snakes (Theorem~\ref{thm:ehrhart}) and a matrix formula for $L_{P_M}(t)$ for snakes in general (Theorem~\ref{thm:count}).
\smallskip

{We then carry on by studying the {\em distributive lattice} structure associated to $P_M$ and its relation with
{\em distributive polytopes} (which are those polytopes whose point set forms a distributive lattice as a sublattice of the componentwise ordering of $\mathbb{R}^n$). In Section~\ref{sec:lattice}, we shall prove  that there exists a bijective affine transformation taking $P_M$ to a full-dimensional {\em distributive integer} polytope $Q_M$ (Theorem~\ref{thm:dpoly}) implying that $L_{P_M}(t)=L_{Q_M}(t)$.}
\smallskip

We use this to make a connection between a natural distributive lattice associated to $Q_M$ (and thus to $P_M$) and their corresponding {\em chain partitioned posets}  (Theorem~\ref{thm:poset}). { As an application, we present a characterization of snakes via the so-called {\em order polytopes} (Theorem~\ref{thm:orderpoly}). We then use the latter and some known results on the Ehrhart polynomial of order polytopes in order to prove unimodality of the $h^*$-vector for two infinite families of snakes (Theorem~\ref{thm:uni1}). This provides new evidence for a challenging conjecture of De Loera, Haws, and K\"oppe~\cite{Loer-09}  which was only known to hold for the class of rank two uniform matroids and for a finite list of examples.} {In Section~\ref{sec:concluding}, we
end by discussing further cases in which this conjecture holds.}



\section{Lattice path matroids}\label{sec:lpm}


A path in the plane is a {\em lattice path}, if it starts at the origin and only does steps of the form $+(1,0)$ and $+(0,1)$ and ends at a point $(m,r)$. One way to encode a lattice path $P$ is therefore  simply to identify it with a vector $\st(P)=(P_1,\dots ,P_{r+m})$, where $P_i\in\{0,1\}$ corresponds to the $y$-coordinate of the $i^{th}$ step of $P$ for all $1\leq i\leq r+m$. The vector $\st(P)$ is the \emph{step vector} of $P$. We will often identify $P$ with its step vector without explicitly mentioning it. Note that since $\st(P)\in\{0,1\}^{r+m}$ the path $P$ can also be identified with a subset of $\{1, \ldots, r+m\}$ or cardinality $r$.
%
Let $L, U$ be two lattice paths from $(0,0)$ to $(m,r)$, such that $L$ never goes above $U$. 
%
%
The {\em lattice path matroid (LPM)} associated to $U$ and $L$ is the 
matroid $M[U,L]$ on the ground set $\{1,\dots ,m+r\}$ whose base set corresponds to all lattice paths from  $(0,0)$ to $(m,r)$ never going below $L$ and never going above $U$.
In~\cite[Theorem 3.3]{Bon-03} it was proved that $M[U,L]$ is indeed a matroid. See Figure~\ref{fig:snake-ex} for an illustration.


\begin{figure}[ht] 
\centering
 \includegraphics[width=1\textwidth]{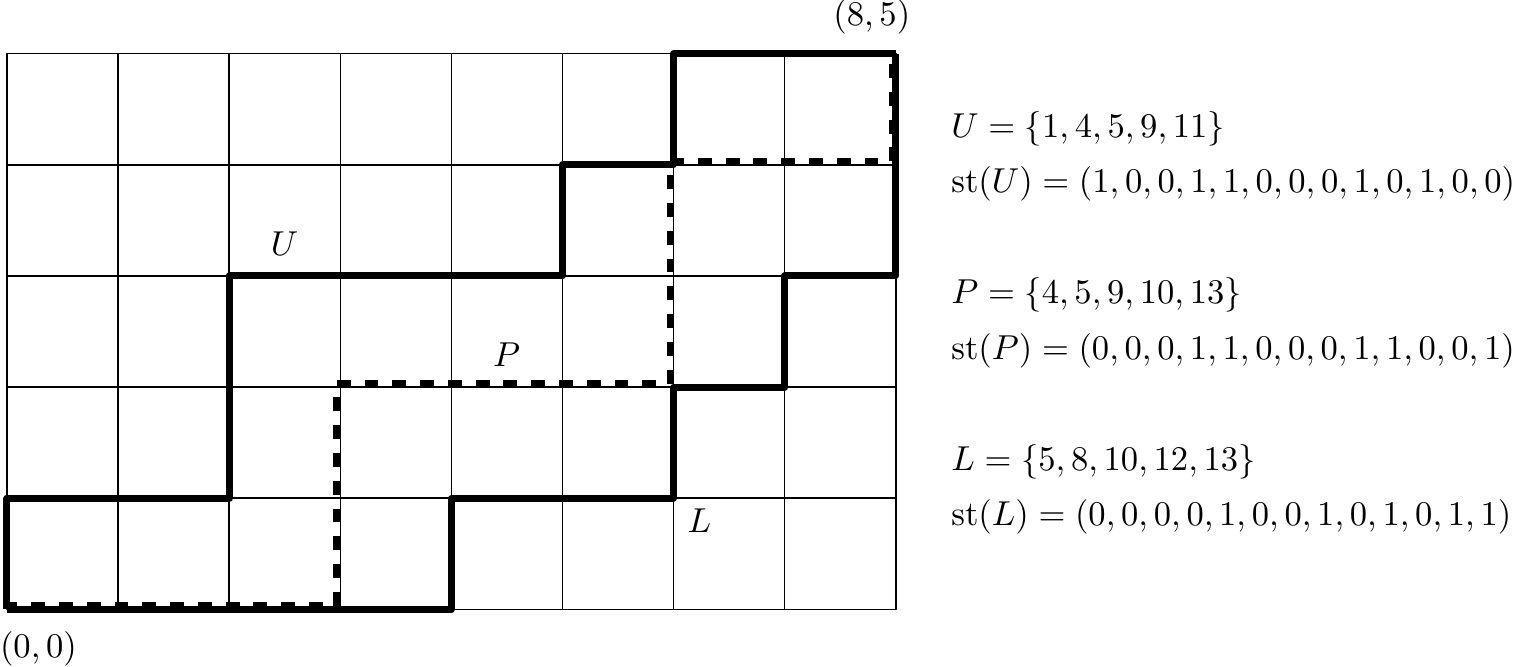}
 \caption{Left: Lattice paths $U$ and $L$ from $(0, 0)$ to $(8, 5)$ and a path $P$ staying between $U$ and $L$ in the diagram of $M[U,L]$. Right: 
 Representations of $U$, $L$, and $P$ as subsets of $\{1,\ldots, 13\}$ and as step vectors. }
 \label{fig:snake-ex}
\end{figure} 

%

It is known~\cite{Bon-03} that the class of LPMs is closed under deletion, contraction, and duality. Indeed, to see the latter, for an LPM $M$, a base in the the dual matroid $M^*$ consists of the 0-steps of a lattice path in the diagram of $M$. Thus, reflecting the diagram of $M$ along the diagonal $x=y$ yields a diagram for $M^*$ and shows that $M^*$ is an LPM as well. See Figure~\ref{fig:duallpm}.

 \begin{figure}[ht] 
\centering
 \includegraphics[width=.5\textwidth]{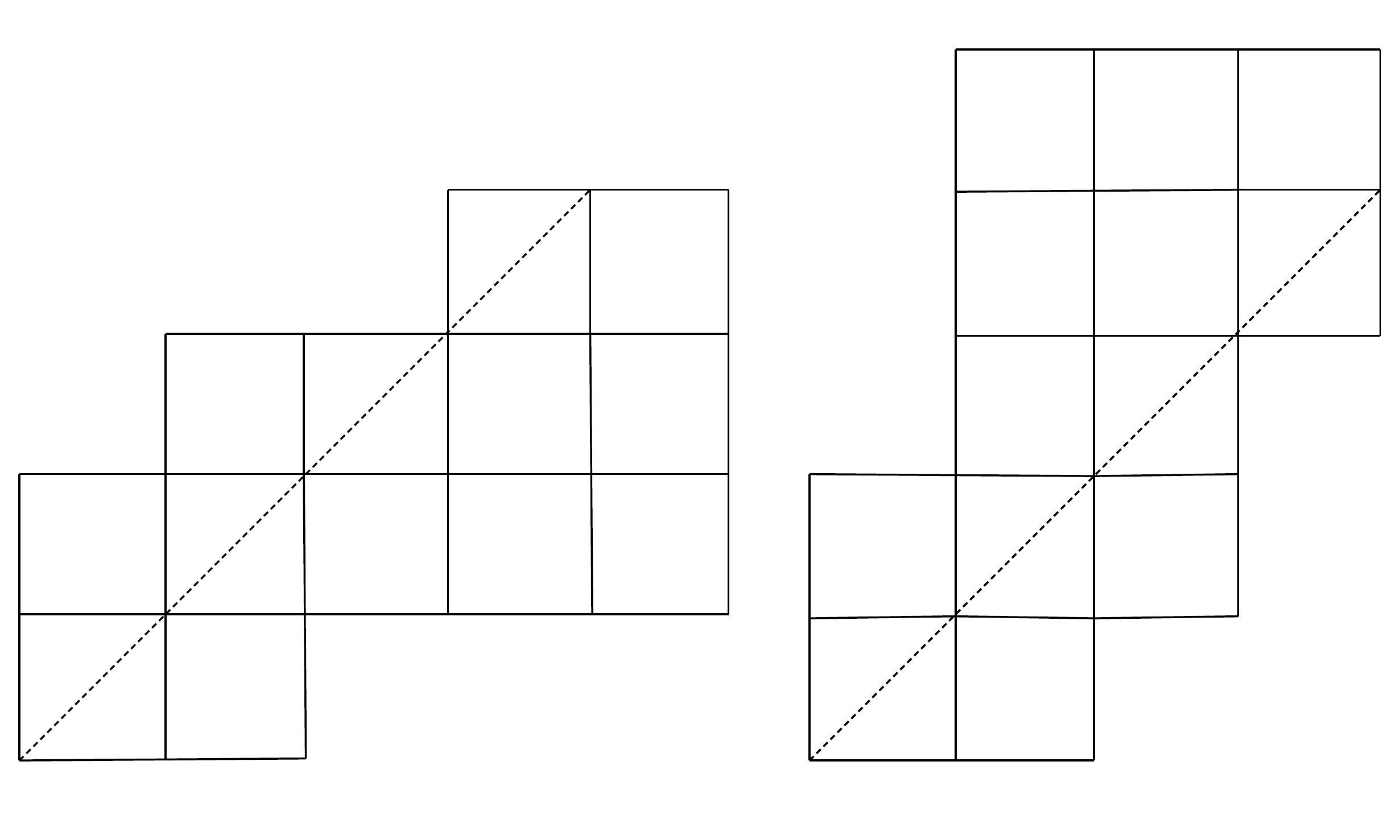}
 \caption{Presentations of an LPM and its dual.}
 \label{fig:duallpm}
\end{figure} 

LPMs are also closed under direct sum. The direct sum in terms of diagrams is illustrated in Figure~\ref{fig:directsumlpm}. In particular, we shall later use the fact (\cite[Theorem 3.6]{Bon-03}) that the LPM $M[U,L]$ is connected if and only if paths $U$ and $L$ meet only at $(0,0)$ and $(m,r)$. 

 \begin{figure}[ht] 
\centering
 \includegraphics[width=.7\textwidth]{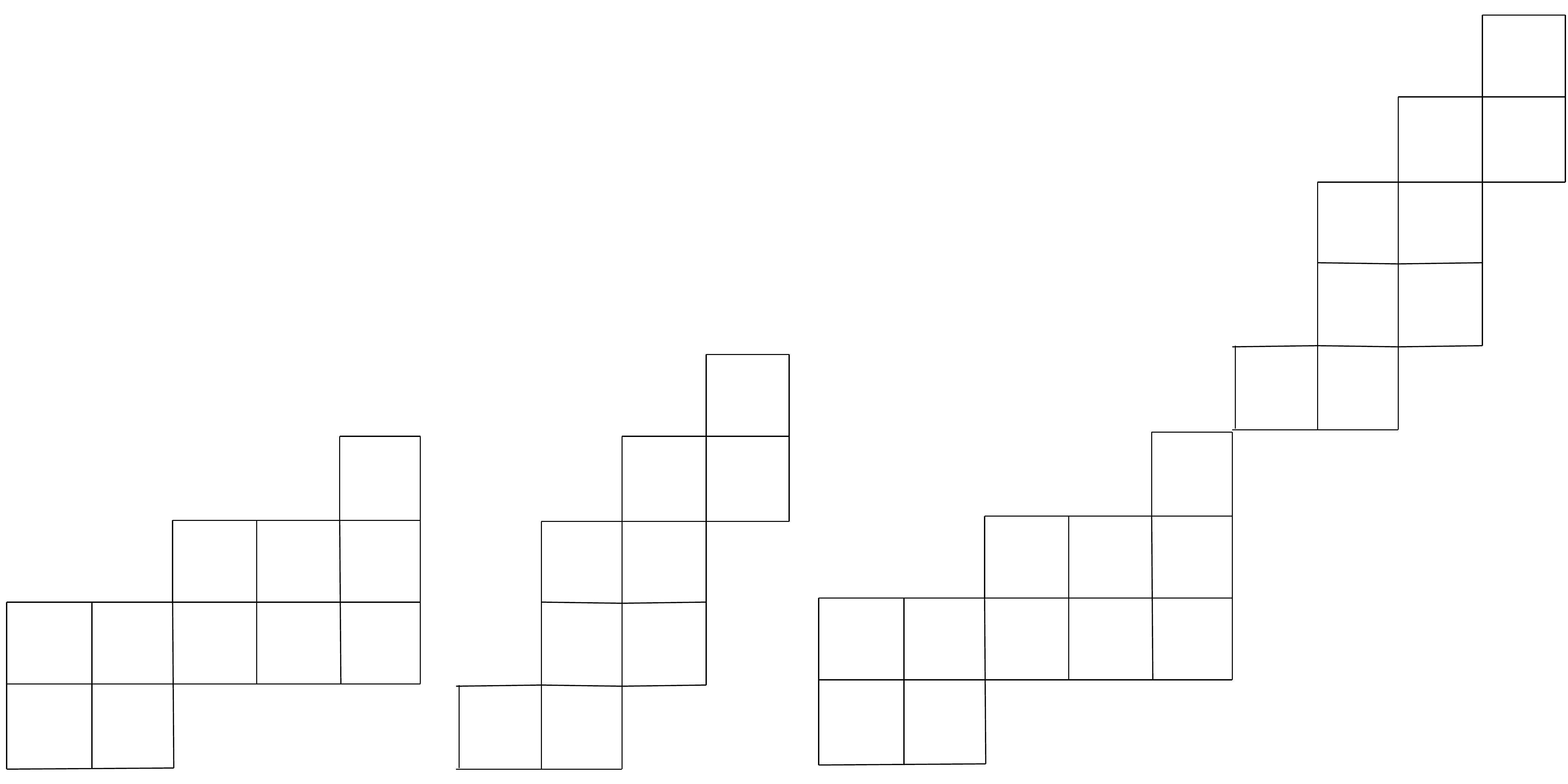}
 \caption{Diagrams of two LPMs and their direct sum.}
 \label{fig:directsumlpm}
\end{figure} 

It is known that if $M=M_1\oplus\dots\oplus M_n$ is the {\em direct sum} of matroids then $P_M=P_{M_1}\times \cdots \times P_{M_n}$ is the {\em Cartesian product} of the corresponding basis polytopes. Since the Ehrhart polynomial of the Cartesian product of two integral polytopes is just the product of their Ehrhart polynomials, it will be enough to work with {\em connected} matroids, that is, matroids that are not the direct sum of two non-trivial matroids. We thus often suppose that $\dim(P_M)=n-1$ where $n$ is the number of elements of $M$.

An LPM is called \emph{snake} if it has at least two elements, it is connected and its diagram has no interior lattice points, see Figure~\ref{fig:snakeslabels}. Note that snakes have also been called {\em border strips} in~\cite{An-17,Bid-12}.

We represent a snake as $S(a_1,a_2,\ldots,a_n)$ if starting from the origin its diagram encloses $a_1\ge 1$ squares to the right, then $a_2\ge 2$ squares up, then $a_3\ge 2$ squares to the right and so on up to $a_n\ge 2$, where the last square counted by each $a_i$ coincides with the first square counted by $a_{i+1}$ for all $i\leq n-1$. 

{An easy  property to check is that Ehrhart polynomial of basis polytopes is invariant under matroid duality. Indeed, if $M^*$ denotes the dual of $M=(E,\bases)$ then $P_{M^*}$ est affinely equivalent to $P_M$ by the isomorphism $x\rightarrow \bar 1-x$ where $\bar 1=(1,\dots ,1)\in\R^E$.
Therefore, unless $M=S(1)=U_{1,2}$, we will suppose that $a_1\ge 2$.}

\begin{figure}[ht] 
\centering
 \includegraphics[width=.65\textwidth]{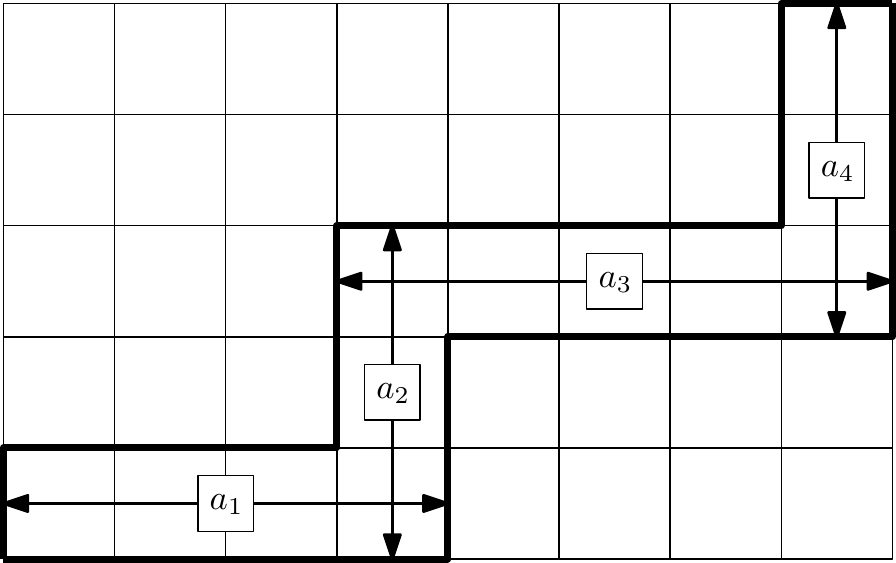}
 \caption{The diagram of the snake $S(a_1,a_2,a_3,a_4)$.}
 \label{fig:snakeslabels}
\end{figure}

The base polytope $P_M$ will be called {\em snake polytope} if $M$ is a snake.

%


\section{Integer points in lattice path matroid polytopes}
\label{sec:basepol}

We shall present a halfspace description of $P_M$ when $M$ is a connected LPM. In order to do this, we give an attractive geometric interpretation of the points in $P_M$ in terms of {polygonal paths}. Let $M=M[U,L]$ be a rank $r$ connected LPM with $r+m$ elements. Let 
$l_i$ be the line defined by $x+y=i$ for each $i=0,\dots ,r+m$ and denote by $R(M[U,L])$ the region bounded by $U$ and $L$. 
Let $T_i=l_i\cap R(M[U,L])$  for each $i=0,\dots ,r+m$, that is, $T_i$ is the segment of $l_i$ contained in $R$.  We notice that the endpoints of $T_i$ are given by the intersection of $l_i$ with $U$ and $L$. Moreover, $T_0=\{(0,0)\}$ and $T_{r+m}=\{(m,r)\}$.
\smallskip

We define a {\em generalized lattice path} $P$
as a polygonal path formed by $r+m$ segments $S_{i+1}(P)$ joining $(x_i,y_i)$ to $(x_{i+1},y_{i+1})$ where $x_i,y_i\in T_i$, $x_i\le x_{i+1}$ and $y_i\le y_{i+1}$ for each $i=0,\dots ,r+m-1$.  Notice that a generalized lattice path is an ordinary lattice path if and only if all its coordinates $(x_i,y_i)$ are integer points.
\smallskip

The points $(x_i,y_i)$ will be called {\em bend points} (points where $P$ may change slope). 
Let $\st(P)=(P_1,\dots ,P_{r+m})$ where $P_{i+1}=y_{i+1}-y_i$ for each $i=0,\dots ,r+m-1$, i.e., $\st(P)$ stores the $y$-steps of the segments $S_i(P)$. 

\begin{examp}\label{ex:gen} We construct the three generalized lattice paths $A,B$ and $C$ in the snake $S(1,2)$ given in
Figure~\ref{fig:genpath}. 

\begin{figure}[ht] 
\centering
 \includegraphics[width=.55\textwidth]{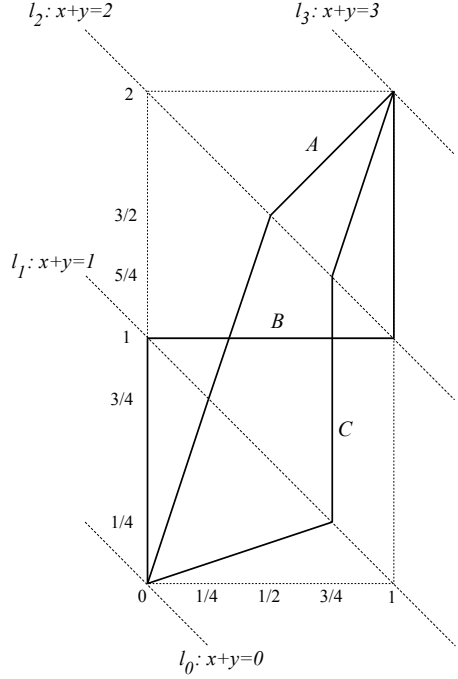}
 \caption{Three generalized lattice paths $A,B$ and $C$ in the snake $S(1,2)$}
 \label{fig:genpath}
\end{figure} 

$A$ is formed by segments : $S_1(A)=\overline{(0,0) (\frac{1}{4},\frac{3}{4})}$, $S_2(A)=\overline{(\frac{1}{4},\frac{3}{4})(\frac{1}{2},\frac{3}{2})}$ and $S_3(A)=\overline{(\frac{1}{2},\frac{3}{2})(1,2)}$, 
\smallskip

$B$ is formed by segments : $S_1(B)=\overline{(0,0) (0,1)}$, $S_2(B)=\overline{(0,1)(1,1)}$ and $S_3(B)=\overline{(1,1)(1,2)}$ ($B$ is an ordinary path corresponding to a base of $S(1,2)$) and
\smallskip

$C$ is formed by segments : $S_1(C)=\overline{(0,0) (\frac{3}{4},\frac{1}{4})}$, $S_2(C)=\overline{(\frac{3}{4},\frac{1}{4})(\frac{3}{4},\frac{5}{4})}$ and $S_3(C)=\overline{(\frac{3}{4},\frac{5}{4})(1,2)}$. 
\smallskip

We have that $\st(A)=(\frac{3}{4}, \frac{3}{4}, \frac{1}{2}), \st(B)=(1,0,1)$ and $\st(C)=(\frac{1}{4}, 1, \frac{3}{4})$.
\end{examp}

\begin{rem}\label{rem:slope}  Let $P$ be a generalized lattice path in the diagram of $M[U,L]$ of rank $r$ and with $r+m$ elements. We have 
\begin{itemize}
 \item[a)] $P$ starts at $(x_0,y_0)=(0,0)$ and ends at $(x_{r+m},y_{r+m})=(m,r)$. 
 \item[b)] $P$ is monotonously increasing since $x_i\le x_{i+1}$ and $y_i\le y_{i+1}$ for each $i=0,\dots ,r+m-1$.
 \item[c)] If $P$ is a lattice path corresponding to a base in $M[U,L]$ then, $P$ is a generalized lattice path where either $x_{i+1}=x_i+1$ and $y_{i+1}=y_i$ or $x_{i+1}=x_i$ and $y_{i+1}=y_i+1$ for each $i=0,\dots ,r+m-1$. We thus have that $\st(P)\in\{0,1\}^{r+m}$ and the notion of step vector for generalized paths generalizes step vectors of ordinary lattice paths.
 \item[d)] {Since $\mathrm{dist}(l_i,l_{i+1})=\frac{1}{\sqrt 2}$ ($l_i$ parallel to the one unit translated line $l_{i+1}$ )} and $x_{i+1}-x_i\ge 0$ (resp. $y_{i+1}-y_i\ge 0$) then $1\ge y_{i+1}-y_i$ (resp. $1\ge x_{i+1}-x_i$).
 \item[e)] Let $\st(P)=(P_1,\dots ,P_{r+m})$. Then, by definition, we clearly have that 
$$\sum_{j=1}^i L_j\leq \sum_{j=1}^i P_j\leq \sum_{j=1}^i U_j \text{ for all } i\in[r+m].$$
In particular, $\sum\limits_{i=1}^{r+m}P_i=r$.
\item[f)] { {If $M[U,L]=M_1\oplus M_2$ is disconnected then its diagram can be obtained from the diagrams of $M_1$ and $M_2$ by identifying the top-right point of $M_1$ with the bottom-left point of $M_2$.} The generalized lattice paths from $M[U,L]$ are the concatenations of genealized lattice paths from $M_1$ and $M_2$.}
\end{itemize}
\end{rem}


Let $\mathcal{C}_M$ be the family of step vectors of all the generalized lattice paths in $M[U,L]$.

\begin{thm}\label{thm:polypaths}
 Let $M=M[U,L]$ be a rank $r$ LPM with $r+m$ elements and let $\st(L)=(L_1,\dots ,L_{r+m})$ and $\st(U)=(U_1,\dots ,U_{r+m})$. Then, $\mathcal{C}_M$ equals
 $$\left\{p\in\mathbb{R}^{r+m}\mid 0\leq p_i\leq 1 \text{ and } \sum_{j=1}^i L_j\leq \sum_{j=1}^i p_j\leq \sum_{j=1}^i U_j \text{ for all } i\in[r+m] \right\}.$$
 
\end{thm}

\begin{proof} Let $\st(P)=(P_1,\dots, P_{r+m})\in \mathcal{C}_M$. By definition $P_{i+1}=y_{i+1}-y_i$ for each $i=0,\dots ,r+m-1$. Thus, $\st(P)$ satisfies the first set of inequalities by Remark~\ref{rem:slope}~(d). By Remark~\ref{rem:slope}~(e) we conclude that $\st(P)\in\R^{r+m}$ satisfies the remaining inequalities.
\smallskip

Let conversely $p\in\mathbb{R}^{r+m}$ such that $0\leq p_i\leq 1$ and $\sum\limits_{j=1}^i L_j\leq \sum\limits_{j=1}^i p_j\leq \sum\limits_{j=1}^i U_j $ for all  $i\in[r+m]$. We consider the points 
$$(x_0,y_0)=(0,0)  \text{ and } (x_i,y_i)=(i-\sum\limits_{j=1}^i p_j,\sum\limits_{j=1}^i p_j) \text{ for each } i=1,\dots ,r+m.$$ 

We clearly have that $y_{i+1}\ge y_i$ for all $i$ since $p_i\ge 0$. Moreover, 
$$x_{i+1}-x_i=(i+1)-\sum\limits_{j=1}^{i+1} p_j - (i-\sum\limits_{j=1}^i p_j)=1-p_{i+1}$$
but $1-p_{i+1}\ge 0$ since $p_i\le 1$.

Now, $(x_i,y_i)\in T_i$,  indeed, $x_i+y_i=i-\sum\limits_{j=1}^i p_j+\sum\limits_{j=1}^i p_j=i$ and thus $(x_i,y_i)$ belongs to line $l_i$. Moreover, $(x_i,y_i)$ belongs to $R(M[U,L])$ since $\sum\limits_{j=1}^i L_j\leq \sum\limits_{j=1}^i p_j\leq \sum\limits_{j=1}^i U_j $ and thus $\sum\limits_{j=1}^i (1-L_j)\geq \sum\limits_{j=1}^i (1-p_j)\geq \sum\limits_{j=1}^i (1-U_j)$.
Therefore, the points $(x_i,y_i)$ form the generalized lattice path $C$ with $\st(C)=(p_1,\dots ,p_{r+m})$.
\smallskip

\end{proof}

\begin{thm}\label{thm:polypaths1}
 Let $M=M[U,L]$ be a rank $r$ LPM on $r+m$ elements. Then, $P_M=\mathcal{C}_M$.
\end{thm}
\begin{proof} We first prove that $P_M\subseteq\mathcal{C}_M$. By Remark~\ref{rem:slope} any base of $M$ corresponds to a generalized lattice path. Therefore, by Theorem~\ref{thm:polypaths}, any vertex of $P_M$ belongs to $\mathcal{C}_M$ and since $\mathcal{C}_M$ is convex (it has a halfspace intersection description) then $P_M\subseteq\mathcal{C}_M$. 
 \smallskip

We now prove that $P_M\supseteq\mathcal{C}_M$. We show that every element in $\mathcal{C}_M$ is a convex combination of step vectors corresponding to ordinary lattice paths. We proceed by induction on the number $n$ of elements of $M$. If $n=1$, then there are only ordinary paths, so we are done. Now suppose $n>1$. If $M=M_1\oplus M_2$ is disconnected, we have $P_{M}=P_{M_1}\times P_{M_2}$, {see also Remark~\ref{remextdiscon}}, and by induction $P_{M_1}\supseteq \mathcal{C}_{M_1}$ and $P_{M_2}\supseteq \mathcal{C}_{M_2}$. This gives $P_{M}\supseteq\mathcal{C}_{M_1}\times \mathcal{C}_{M_2}$ where the latter consists of step vectors of concatenated generalized lattice paths in $M_1$ and $M_2$. {Thus, by Remark~\ref{rem:slope}~f), we may obtain}, $\mathcal{C}_{M_1}\times \mathcal{C}_{M_2}=\mathcal{C}_{M}$

Suppose now that $M$ is connected.
Note that $\mathcal{C}_M$ is contained in the $(n-1)$-dimensional subspace $H$ defined by the equality 
$$\sum\limits_{j=1}^{r+m} L_j=\sum\limits_{j=1}^{r+m} p_j=\sum\limits_{j=1}^{r+m} U_j=r.$$ Let $C$ be a point on the boundary of $\mathcal{C}_M$ with respect to $H$. Hence, $C$ satisfies the equality in one of the inequalities of the halfspace description of $\mathcal{C}_M$ of Theorem~\ref{thm:polypaths}. If $C_i=0$ for some $i\in[n]$, then we can consider $M\setminus i$, which corresponds to all lattice paths of $M$ with a $0$ in the $i^{th}$ coordinate. By induction $C$ is in the convex hull of these vectors. The dual argument works if $C_i=1$. If $\sum_{j=1}^i L_j= \sum_{j=1}^i C_j$, we know that $\sum_{j=1}^i C_j<\sum_{j=1}^i U_j$ since the case of both equalities cannot happen in a connected $M$.
Thus, $C$ coincides with $L$ at a point $(x,y)\in\mathbb{Z}^2$ which is not in $U$. We consider the lattice path matroid $M'$ with lower path $L$ and upper path $U'$, where $U'$ arises from $U$ by going right from the point $(x',y)\in U$ until reaching $(x,y)$ and then up until reaching the point $(x,y')\in U$. Now, $M'$ contains $C$ and is not connected. By applying the above argument for disconnected LPMs we get $C\in P_{M'}$, but clearly we have $P_{M'}\subseteq P_M$. If $C$ coincides with $U$ at a point which is not in $L$ the analogous argument works. We thus have shown that all points on the boundary of $\mathcal{C}_M$ are in $P_M$, that is,  $P_M\supseteq\mathcal{C}_M$.
\end{proof}

\begin{rem}\label{rem:nexdesc} The above geometric description of the points in $P_M$ seems to be new as far as we are aware. The equality 
$$P_M=\left\{p\in\mathbb{R}^{n}\mid \forall i\in[n]: 0\leq p_i\leq 1; \sum_{j=1}^i L_j\leq \sum_{j=1}^i p_j\leq \sum_{j=1}^i U_j\right\}$$
 was already stated in~\cite[Lemma 3.8]{Bid-12} but it seems that the given proof contains a wrong argument. Indeed, in the proof a vector $B:=(a-a_iX)/(1-a_i)$ is defined, where $a$ is a vector satisfying the inequalities describing $\mathcal{C}_M$ without any coordinate in $\{0,1\}$, $a_i$ a smallest entry of $a$, and $X_i$ a $0,1$-step vector of an ordinary lattice path such that $X_i=1$. It is claimed that $B$ also satisfies the inequalities of $\mathcal{C}_M$, in particular verifying that $0\le B_i\le 1$ for all $i$. However, if $a_j=1-\frac{a_i}{2}$ and $X_j=0$, then $B_j=a_j/(1-a_i)=(1-\frac{a_i}{2})/(1-a_i)>1$.
\end{rem}

Let $\mathcal{C}^k_M$ be the family of step vectors of all the generalized lattice paths $P$ of $M[U,L]$ such that all the bend points $(x,y)$ of $P$ satisfy $kx,ky\in\mathbb{Z}$. The following two corollaries can be easily deduced from Theorems~\ref{thm:polypaths} and~\ref{thm:polypaths1}.

\begin{cor}\label{cor:integ}
Let $M$ be an LPM on $n$ elements and let $k\in\mathbb{N}$. Then, a point $p\in\mathbb{R}^n$ is in $kP_M\cap\mathbb{Z}^n$ if and only if $p$ corresponds to a generalized lattice path in  
$\mathcal{C}^k_M$.
\end{cor}

\begin{cor}\label{cor:inter}

 Let $M=M[U,L]$ be a rank $r$ LPM on $r+m$ elements. Then, a point $p\in\mathbb{R}^{r+m}$ is in the interior of $P_M$ if and only if $p$ corresponds to a generalized lattice path $P$ of $M$ such that
 \begin{itemize}    
\item $P\cap U=P\cap L=L\cap U$,
\item $P$ is strictly monotone, that is, $x_i< x_{i+1}$ and $y_i< y_{i+1}$ for all $i$.
\end{itemize}
\end{cor}

\subsection{Application: a formula for $L_{P_{S(a,b)}}(t)$}
In~\cite{Loer-09}, $L_{P_M}(t)$ is explicitly calculated  for $28$ selected matroids $M$ and in~\cite{Katz-05}
for all uniform matroids. The following result provides an explicit formula for $L_{P_{S(a,b)}}(t)$.

\begin{thm}\label{thm:ehrhart} Let $a,b\ge 2$ be integers. Then, the Ehrhart polynomial of $P_{S(a,b)}$ is given by
	\[
		1+\frac{1}{(a-1)!(b-1)!} \sum_{i=1}^{a+b-1} \left(\sum_{j=i-1}^{a+b-2} (-1)^{j-i+1} \cdot\frac{B_{j-i+1}\cdot\sigma_{a+b-2-j}}{j+1} \binom{j+1}{i}\right) t^i.
	\]
	
\noindent where $\sigma_\ell$ is the $\ell^{th}$ symmetric function on the numbers $$1,2,\ldots,a-1,1,2,\ldots,b-1$$ and $B_m$ is the $m^{th}$ Bernoulli number.
\end{thm}

In order to prove this, we will state a couple of general results for counting the number of integer points in $kP_{S(a_1,\ldots,a_n)}$. For this purpose, we will use the description given by Corollary~\ref{cor:inter}, so we shall focus on counting the number of generalized lattice paths whose coordinates are integer multiples of $\frac{1}{k}$. Later on we will see how these results simplify in the case $n=2$ and yield the formula in Theorem~\ref{thm:ehrhart}.

We begin with the following simple combinatorial lemma. We refer the reader to Figure~\ref{fig:excount} for a geometric interpretation that relates the lemma to generalized lattice paths.

\begin{figure}[htbp]
	\centering
		\includegraphics[width=0.80\textwidth]{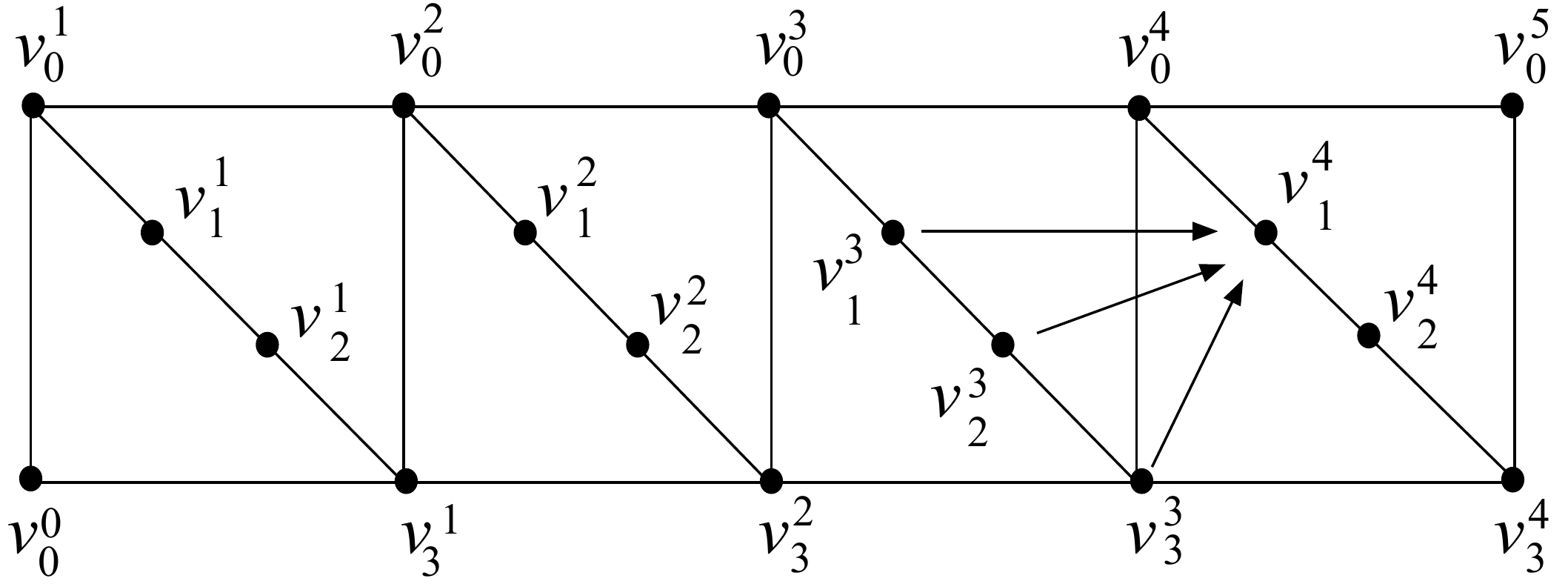}
 \caption{General construction.}\label{fig:excount}
\end{figure}

\begin{lem}\label{lem:recursion}
	Let $k$ and $a$ be positive integers. Let $u^1=(u_0^1,\ldots,u_k^1)$ be a vector of positive integers. For $i= 1,\ldots,a-1$ define recursively the entries of vectors $u^i$ as follows:
	
	\[
		u_j^{i+1}=\sum_{\ell=j}^k u_\ell^{i}, \quad \text{ for } j\in\{0,1,\ldots,k\}.
	\]
	
	Then $u^a=A(k,a)u^1$ where $A(k,a)$ is the $(k+1)\times (k+1)$ matrix with entries $A_{ij}$ given by:
	
	\[
		A_{ij}=\binom{a-2+j-i}{j-i}.
	\]
	
\end{lem}

\begin{proof}
	Consider the $(k+1)\times(k+1)$ matrix $J$ with entries $J_{ij}$ given by
	
	\[
	J_{ij}=
		\begin{cases}
			1\quad \text{ for $j\geq i$}\\
			0\quad \text{ for $j<i$}.
		\end{cases}
	\]
	
	Using this matrix the recursion can be stated as $u^{i+1}=Ju^{i}$. Therefore, $u^{a}=J^{a-1}u^1$. An easy induction argument shows that $J^{a-1}=A(k,a)$.
\end{proof}

Suppose now that we want to count the number of generalized lattice paths in the snake $S(a_1,a_2,\ldots,a_n)$ whose coordinates are integer multiples of $\frac{1}{k}$. A standard technique is to count the number of paths recursively starting in the lower left corner and then writing at each bend point $p$ the number of possible ways to get to that point. The number in $p$ equals to the sum of the numbers in the possible points that precede $p$.

\begin{figure}[htbp]
	\centering
		\includegraphics[width=0.75\textwidth]{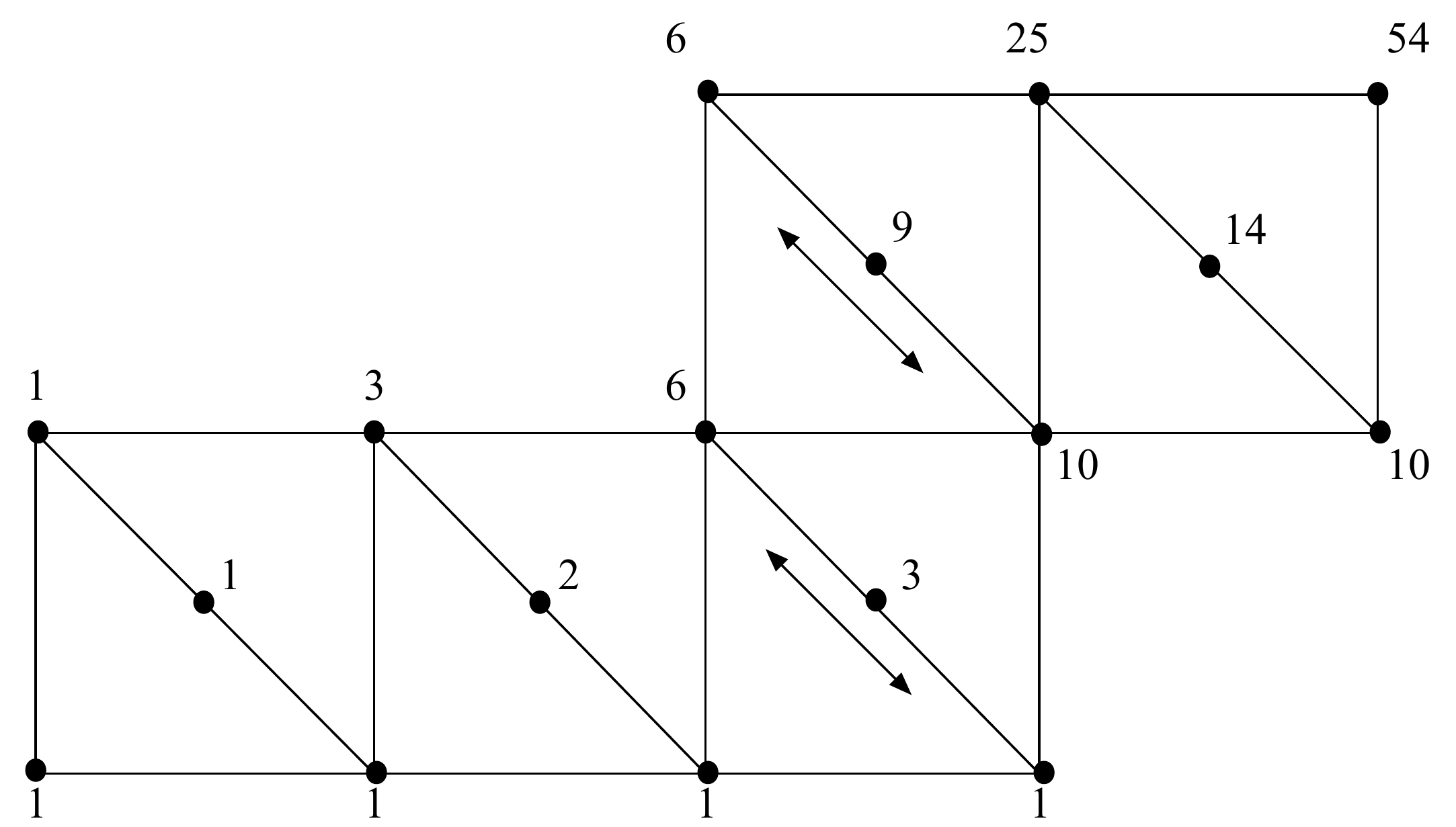}
 \caption{Counting bends horizontally and vertically.}	\label{fig:ExBends}
\end{figure}

As long as the snake is horizontal, the process above yields the recursion in Lemma~\ref{lem:recursion}: we add the numbers in points before and below $p$. However, whenever the snake bends the numbers in the diagonal play an ``inverse role'' and we have to invert the recursion accordingly, see Figure~\ref{fig:ExBends}. If we proceed inductively on the number of bends in the snake we obtain the following result.

\begin{thm}\label{thm:count} Let $k\ge 1$ be an integer. Then,
the number of integer points in $kP_{S(a_1,\ldots,a_n)}$ is $$u^T \left(\prod_{j=1}^n A(k,a_i)R\right) u,$$	where $u=(1,1,\ldots,1)$ is the vector
with $k+1$ ones, $R$ is the matrix that inverts the coordinates of a vector and the matrices $A(k,a_i)$ are defined as in Lemma~\ref{lem:recursion}.
\end{thm}

When $n=2$ we have a snake with just one bend, say $S(a,b)$. This allows the formula above to be simplified to a polynomial in $k$.

\begin{proof}[Proof of Theorem~\ref{thm:ehrhart}]
	Theorem~\ref{thm:count} states that the number of lattice points in the dilation $kP_{S(a,b)}$ is equal to 
	\[
		u^TA(k,b)RA(k,a)u.
	\]
	
	First, the $j^{th}$ entry of $A(k,a)u$ is:
	
	\[
		\sum_{i=0}^{k+1-j} \binom{a-2+i}{i}= \binom{a+k-j}{k+1-j}.
	\]
	
	When we invert the coordinates, we get the vector
	
	\[
		\left(\binom{a-1}{0}, \binom{a}{1}, \ldots, \binom {a+k-1}{k}\right).
	\]
	
	Now we have to multiply from the left by $u^T A(k,b)$ or, equivalently, multiply by $A(k,b)$ and sum the coordinates of the obtained vector. After multiplying we can arrange the sum of entries as follows:
	
	\begin{align*}
		\sum_{j=0}^k \sum_{i=0}^j &\binom{a-1+j}{j} \binom{b-2+i}{i} =\sum_{j=0}^k \binom{a-1+j}{j} \sum_{i=0}^j \binom{b-2+i}{i}\\
		&= \sum_{j=0}^k \binom{a-1+j}{j} \binom{b-1+j}{j}\\
		&= \sum_{j=0}^k \binom{a-1+j}{a-1} \binom{b-1+j}{b-1}\\
		&= \frac{1}{(a-1)!(b-1)!} \sum_{j=0}^k \left((a-1+j)\cdots(j+1)\right)\left((b-1+1)\ldots(j+1)\right)\\
		&= \frac{H(0)+H(1)+\ldots+H(k)}{(a-1)!(b-1)!}.
	\end{align*}

  In the last line $H$ is the polynomial
	
	\[
	  H(t)=(t+1)(t+2)\cdots(t+a-1)(t+1)(t+2)\cdots(t+b-1).
	\]
	
	The proof ends by expanding the polynomial using symmetric functions on the multiset $\{1,2,\ldots,a-1,1,2,\ldots,b-1\}$, grouping terms with the same $t$ exponent, using standard formulas for the sums of first powers and regrouping as a polynomial in $k$.

\end{proof}

\section{Distributive lattice structure}
\label{sec:lattice}

In this section we will study a lattice structure induced by lattice path matroids.  For more on posets and lattices we refer the reader to~\cite{Dav2002}.

\subsection{Distributive polytopes}
A polytope $P\subseteq \mathbb{R}^n$ is called \emph{distributive} if for all $x,y\in P$ also their componentwise maximum and minimum $\max(x,y)$ and $\min(x,y)$ are in $P$, see Figure \ref{fig:dist-polex}.
The point set of these polytopes forms a distributive lattice as a sublattice of the componentwise ordering of $\mathbb{R}^n$. The latter is a distributive lattice itself with join and meet operations  componentwise maximum and minimum, respectively. Distributive polytopes have been characterized combinatorially and geometrically in~\cite{Fel-11}.

\begin{figure}[ht] 
\centering
 \includegraphics[width=.46\textwidth]{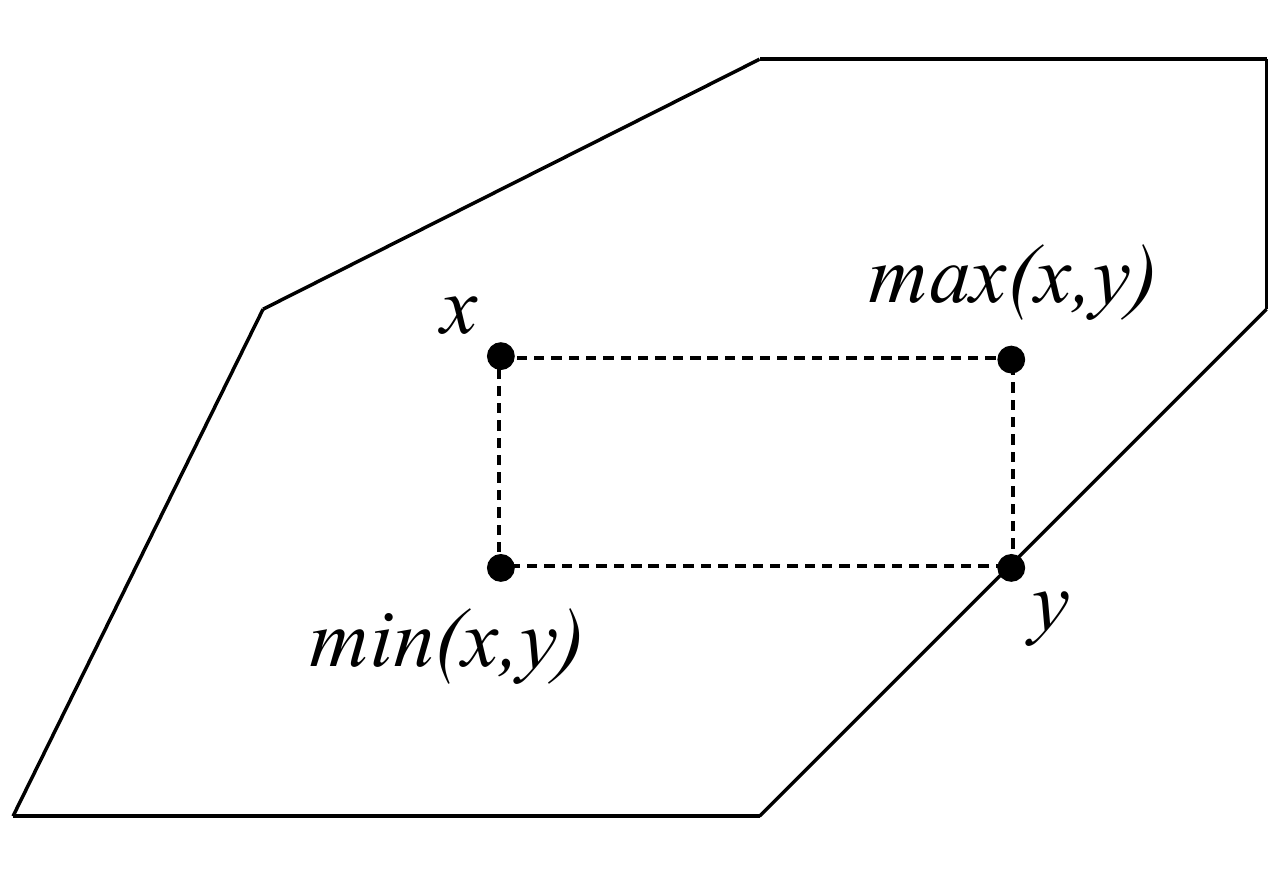}
 \caption{A distributive polytope in $\R^2$}
 \label{fig:dist-polex}
\end{figure} 

The following relates $P_M$ with distributive polytopes when $M$ is a LPM.

\begin{thm}\label{thm:dpoly}
Let $M=M[U,L]$ be a rank $r$ connected LPM on $r+m$ elements. Then, there exists a bijective affine transformation taking $P_M\subset\mathbb{R}^{r+m}$ into a full-dimensional distributive integer polytope $Q_M\subset\mathbb{R}^{r+m-1}$ consisting of all 
$q\in\mathbb{R}^{r+m-1}$ such that 

$$\begin{array}{lcl}
0\leq (-1)^{L_{i+1}}(q_{i+1}-q_i)\leq 1 & ~ & \text{ for all } i\in[r+m-2]\\
~&\text{and}&~\\
0\leq q_i\leq \sum\limits_{j=1}^i(U_j-L_j) & ~ &  \text{ for all } i\in[r+m-1].
\end{array}$$     

{And therefore,} $L_{P_M}(t)=L_{Q_M}(t)$. 
\end{thm}
\begin{proof}
Let $\st(L)=(L_1,\dots ,L_{r+m})$ and $\st(U)=(U_1,\dots ,U_{r+m})$. Then, by Theorems~\ref{thm:orderpoly} and~\ref{thm:polypaths1}, we have that 
 $$P_M=\left\{p\in\mathbb{R}^{r+m}\mid 0\leq p_i\leq 1 \text{ and } \sum_{j=1}^i L_j\leq \sum_{j=1}^i p_j\leq \sum_{j=1}^i U_j \text{ for all } i\in[r+m] \right\}.$$
 

Let
$$\begin{array}{llll}
\pi : & P_M\subset \mathbb{R}^{r+m} & \longrightarrow & \mathbb{R}^{r+m-1}\\
 & p=(p_1,\dots ,p_{r+m}) & \mapsto & (p_1-L_1,\dots , \sum_{j=1}^{r+m-1} (p_j-L_j))\\
\end{array}$$

We thus have that $\pi$ is an affine mapping consisting of a translation by $-\st(L)$ and of the linear map using the above halfplane description of $P_M$. Clearly, $\pi$ is injective. Let $p$ be a point in $P_M$ and let $P$ the corresponding generalized path of $M$ with $\st(P)=p$.
Let $\pi(p)=\pi(p_1,\dots ,p_{r+m})=(q_1,\dots ,q_{r+m-1})$.  By Remark~\ref{rem:slope} (e) we have both 
$$\sum\limits_{j=1}^i L_j\leq \sum\limits_{j=1}^i P_j \text{ for all }i\in[r+m] \text{ and thus } q_i=\sum\limits_{j=1}^i (P_i-L_j)\ge 0$$
and 
$$\sum_{j=1}^i P_j\leq \sum\limits_{j=1}^i U_j \text{  for all } i\in[r+m] \text{ and thus } q_i=\sum\limits_{j=1}^i (P_i-L_j)\le \sum\limits_{j=1}^i (U_i-L_j).$$ 

Now, 
$$q_{i+1}-q_i=\sum\limits_{j=1}^{i+1}( p_j-L_j)-\sum\limits_{j=1}^{i}( p_j-L_j)=p_{i+1}-L_{i+1}.$$

Since $L_{i+1}=0$ or $1$ and $0\leq p_{i+1}\leq 1$ then we clearly have that  $-1\le p_{i+1}-L_{i+1}\le 0$ if $L_{i+1}=1$ and $0\le p_{i+1}-L_{i+1}\le 1$ if $L_{i+1}=1$.

Therefore, $\pi(P_M)$ is a polytope contained in a polytope $Q_M$ having the following description 

$$Q_M=\{q\in\mathbb{R}^{r+m-1}\mid  0\leq (-1)^{L_{i+1}}(q_{i+1}-q_i)\leq 1 \text{ for all } i\in[r+m-2] \text{ and }$$
$$ \hspace{4cm} 0\leq q_i\leq \sum\limits_{j=1}^i(U_j-L_j) \text{ for all } i\in[r+m-1]\}.$$
 
Conversely, it is easy to see that $Q_M\subseteq \pi(P_M)$ by constructing for a given $q\in Q_M$ a preimage $p$ in $P_M$ under $\pi$ by setting $p_1=q_1$ and $p_i=q_i-p_{i-1}+\sum\limits_{j=1}^iL_j$ for $1<i<r+m$ and $p_{r+m}=r-\sum\limits_{j=1}^{r+m-1}p_j$.
 
By using the above description or the characterization in~\cite{Fel-11}, it is straight-forward to verify that $Q_P$ is closed under componentwise maximum and minimum. Therefore, $Q_P$ is a distributive polytope. 
\smallskip

Furthermore, since $\dim(P_M)=r+m-1$ then $Q_M$ is full-dimensional.  Let $k$ be a positive integer. It is also immediate that $\pi$ sends points in $\frac{1}{k}\mathbb{Z}^{r+m}$ to $\frac{1}{k}\mathbb{Z}^{r+m-1}$. Indeed, if $\pi$ would send a point $p\in P\setminus\frac{1}{k}\mathbb{Z}^{r+m}$ to $\frac{1}{k}\mathbb{Z}^{r+m-1}$, then it would exist a minimal index $i$ such that $p_i\notin\frac{1}{k}\mathbb{Z}$ and since $\sum\limits_{j=1}^{i-1} p_j\in \frac{1}{k}\mathbb{Z}$ we would have $\sum\limits_{j=1}^{i-1} p_j + p_i\notin \frac{1}{k}\mathbb{Z}$ which would be a contradiction. {And thus, $L_{P_M}(t)=L_{Q_M}(t)$}. 
\end{proof}

\begin{rem} \label{remextdiscon} Theorem \ref{thm:dpoly} can be extended to a disconnected LPM $M$ as follows. Let  $M=M_1\oplus\dots\oplus M_c$ where $M_i$ is a connected LPM on $e_i=r_i+m_i$ elements with $r_i$ and $m_i$ the number of lines (rank) and columns in the presentation of $M_i, i=1,\ldots ,c$.  Thus, $M=M[U,L]$ is an LPM with $c$ connected components and its representation consists of identifying the top-right corner of  $M_i$ with the bottom-left corner of $M_{i+1}$ for each $i=1,\dots ,c-1$, see Figure~\ref{fig:directsumlpm}. We obtain that $L=L_1, \ldots , L_c$ and $U=U_1, \ldots , U_c$ 
and thus $M[U,L]$ is of rank $r=\sum\limits_{i=1}^c r_i$ having $n=\sum\limits_{j=1}^c (r_j+m_j)$ elements. 

We clearly have that $\sum\limits_{j=1}^i L_i < \sum\limits_{j=1}^i U_i$ for all $i\in [n]$ except  at the values $i=r_1+m_1, r_1+m_1+r_2+m_2,\dots , \sum\limits_{j=1}^c (r_j+m_j)$, that is, except at the $c$ points where the paths $U$ and $L$ meet (other than $(0,0)$). 

Theorem~\ref{thm:dpoly} can be generalized to $M$ by considering the map

$$\begin{array}{llll}
\psi : & P_M\subset \mathbb{R}^{r+n} & \longrightarrow & \mathbb{R}^{r+n-c}\\
 & p=(p_1,\dots ,p_{r+n}) & \mapsto & \psi (p)=(\psi_1,\dots ,\psi_{r+n-c})\\
\end{array}$$

where $\psi_i$ is the $i^{th}$ nonzero coordinate of $q=(q_1,\dots ,q_{r+n})$ where $q_i=\sum_{j=1}^i (p_j-L_j)$ for all $i\in[n]$ with $\sum\limits_{j=1}^i L_i < \sum\limits_{j=1}^i U_i$ and zero otherwise (that is, $q_i=0$ when $i=r_1+m_1, r_1+m_1+r_2+m_2,\dots , \sum\limits_{j=1}^c (r_j+m_j)$).

Notice that if $M$ is connected then $U$ and $L$ meet only at one point (other than $(0,0)$) and $\psi$ becomes the map $\pi$ given in Theorem~\ref{thm:dpoly}. 
\end{rem}

\begin{examp} Let us consider the snake $S(1,2)$. Since $S(1,2)$ consists of 3 elements and it is connected then  $\dim(P_{S(1,2)})=3-1=2$. We have

$P_{S(1,2)}=\{p\in\mathbb{R}^{3}\mid 0\leq p_i\leq 1 \text{ and } 0\le p_1 \le 1; $

$\hspace{3.2cm}1\le p_1+p_2 \le 2; 2\le p_1+p_2+p_3 \le 2\}.$

We notice that the vertices of $P_{S(1,2)}$ correspond to the three bases $u=\st(U)=(1,1,0), b=\st(B)=(1,0,1)$ and $l=\st(L)=(0,1,1)$, see Figure~\ref{fig:dist-pollex}. 

\begin{figure}[ht] 
\centering
 \includegraphics[width=.7\textwidth]{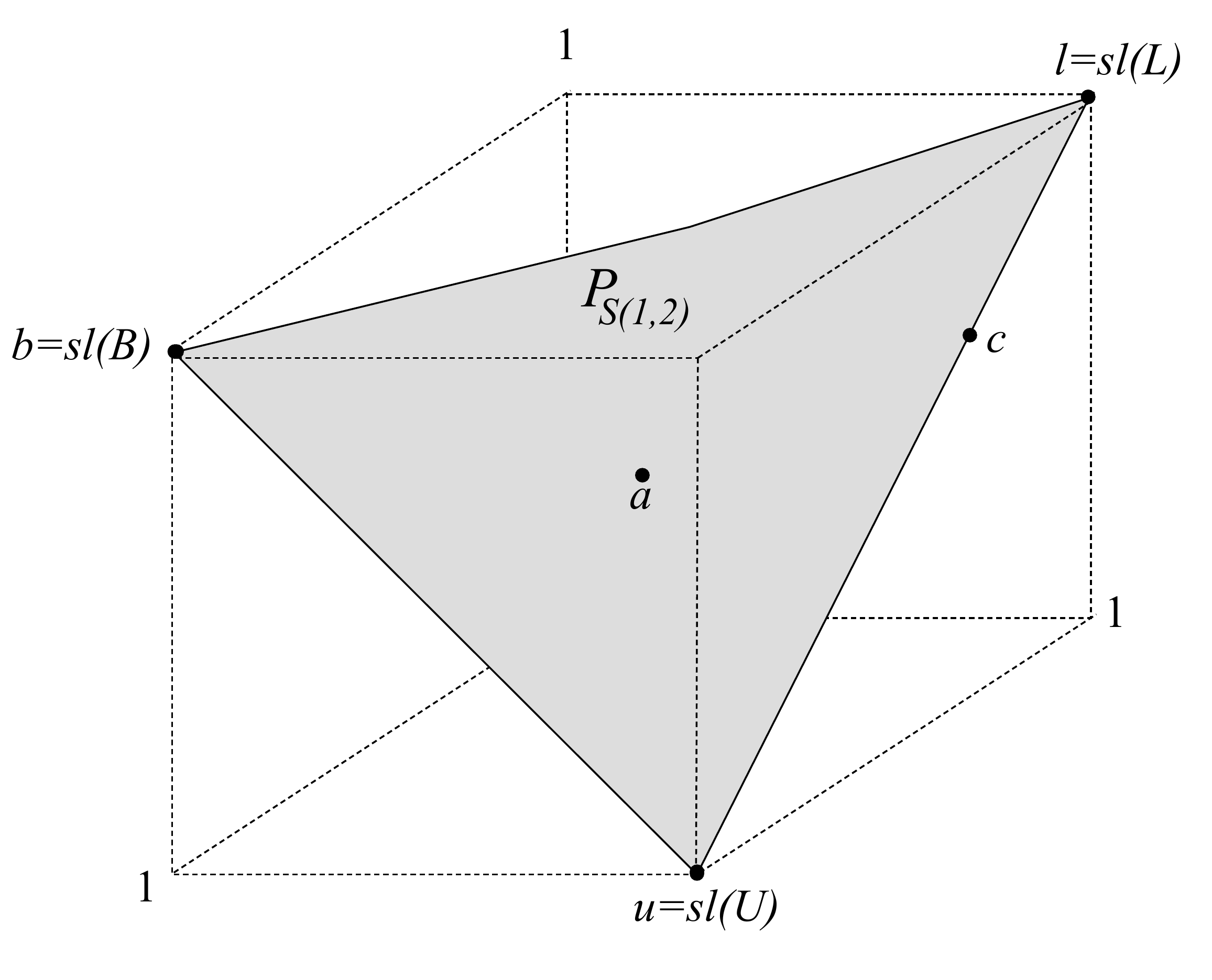}
 \caption{$P_{S(1,2)}$ where the three points $a,b,c$ correspond to the step vectors of the generalized paths $A,B,C$ given in Example~\ref{ex:gen}.}
 \label{fig:dist-pollex}
\end{figure}

Since $\pi(\st(U))=(1,1)$ and $\pi(\st(L))=(0,0)$  then
$$Q_{S(1,2)}=\{q\in\mathbb{R}^{2}\mid  0\leq q_{1}-q_2\leq 1 \text{ and } 0\le q_1,q_2\le 1\}.$$

$Q_{S(1,2)}$ is illustrated in Figure~\ref{fig:dist-pollex1}. We can check that $\pi(a)=(\frac{3}{4}-0,\frac{6}{4}-1)=(\frac{3}{4},\frac{1}{2}), \pi(b)=(1-0,1-1)=(1,0)$ and $\pi(c)=(\frac{1}{4}-0,\frac{5}{4}-1)=(\frac{1}{4},\frac{1}{4})$.

 \begin{figure}[ht] 
\centering
 \includegraphics[width=.49\textwidth]{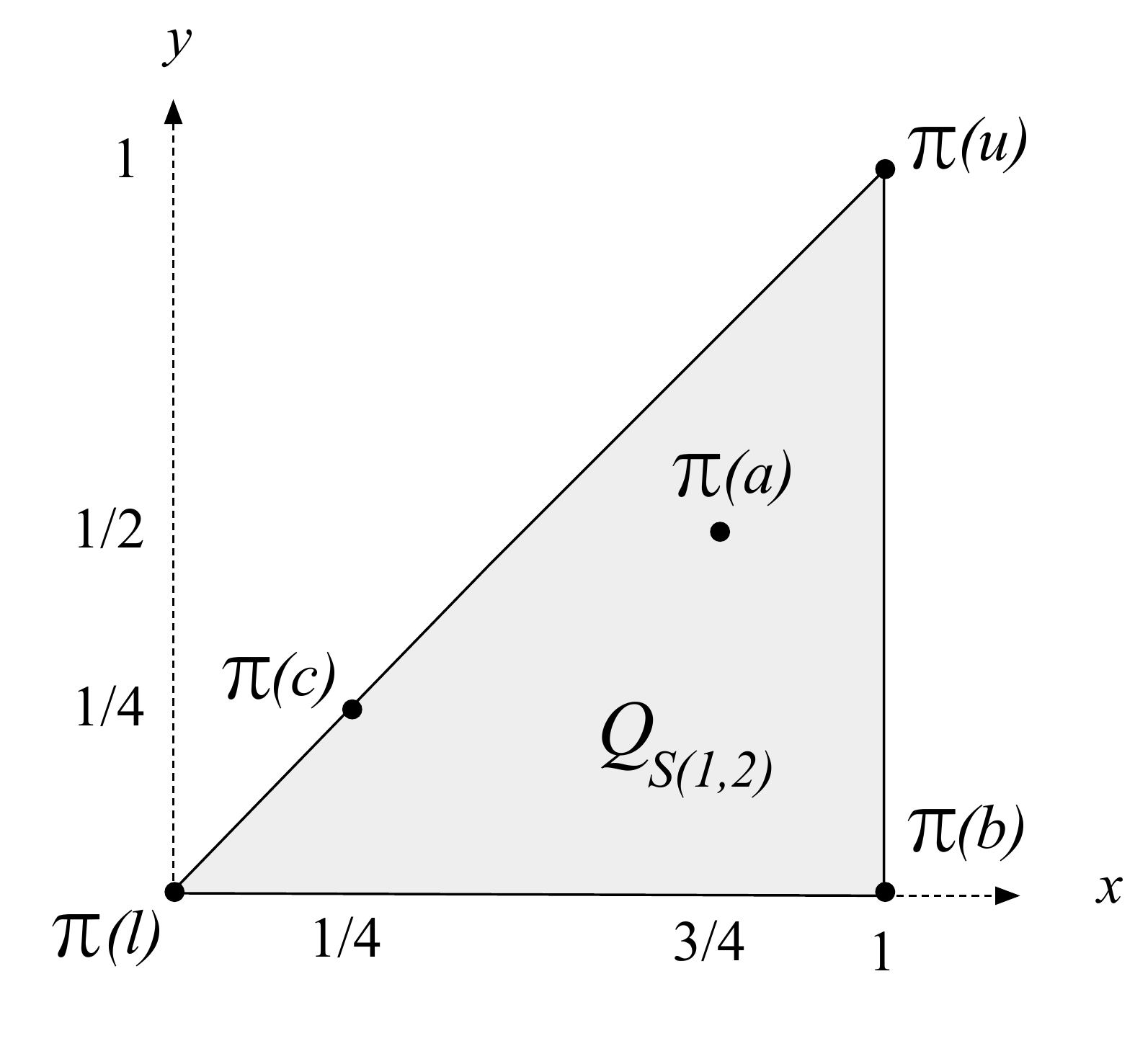}
 \caption{Polytope $Q_{S(1,2)}$.}
 \label{fig:dist-pollex1}
\end{figure}
\end{examp}

\begin{examp} Let us consider the snake $S(2,2)$ (``inverted L'').  Notice that $S(2,2)$ is given by $U_{2,4}$ from which the base $\{1,2\}$ is deleted.  It is known that $P_{U_{2,4}}$ is the octahedron and thus $P_{S(2,2)}$ is the pyramid obtained from removing the north vertex of the octahedron. 
\smallskip

Since $\st(U)=(1,0,1,0)$ and $\st(L)=(0,0,1,1)$ then $\pi(\st(U))=(1-0,1-0,2-1)=(1,1,1)$ and $\pi(\st(L))=(0-0,0-0,0-0)=(0,0,0)$ and thus
$$Q_{S(2,2)}=\{q\in\mathbb{R}^{3}\mid  0\leq q_{2}-q_1\leq 1, 0\leq q_{2}-q_3\leq 1 \text{ and } 0\le q_1,q_2,q_3\le 1\}.$$
 See Figure~\ref{fig:dist-pol} for an illustration.
  
\begin{figure}[ht] 
\centering
 \includegraphics[width=.5\textwidth]{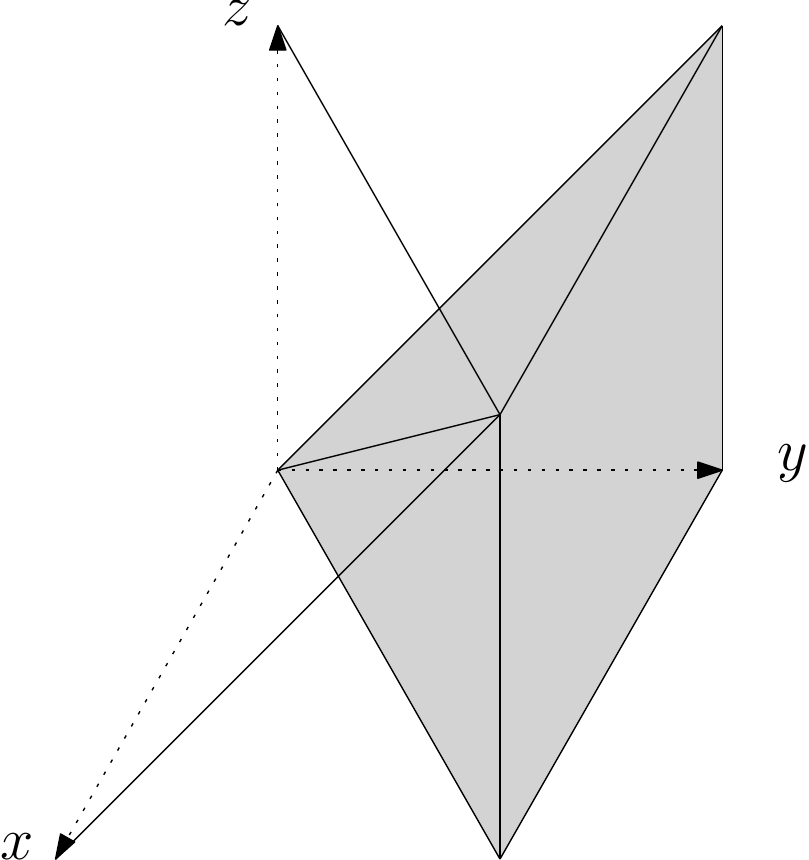}
 \caption{The distributive polytope $Q_{S(2,2)}$.}
 \label{fig:dist-pol}
\end{figure} 
\end{examp}


We are able to determine a distributive lattice structure in $P_M$ by using the above application $\pi$, i.e., for $p,p'\in P_M$ we have $p\leq p'$ if and only if $\pi(p)\leq \pi(p')$ with respect to the componentwise ordering in $Q_M$. In particular, since the set $\frac{1}{t}\mathbb{Z}^{r+m-1}$ is closed under componentwise minimum and maximum then $\pi$ restricts to a distributive lattice on $tQ_M\cap\mathbb{Z}^{r+m-1}$ which carries over to $tP_M\cap\mathbb{Z}^{r+m}$. This leads to:

\begin{cor}\label{cor:dl} Let $M=M[U,L]$ be a rank $r$ connected LPM on $r+m$ elements and let $k$  be a positive integer. Then, both the points in $P_M$ and the points in $kP_M\cap\mathbb{Z}^{r+m}$ carry a distributive lattice structure defined by 
 $$p\leq p' \text{ if and only if } \sum\limits_{j=1}^i p_j\leq \sum\limits_{j=1}^i p'_j \text{ for all } i\in [r+m].$$ Moreover, in $\mathcal{C}_M$ and $\mathcal{C}^k_M$ this order corresponds to one generalized lattice path lying above another generalized lattice path.
 
\end{cor}

\subsection{Chain partitioned poset}
We study further the distributive lattices given in Corollary~\ref{cor:dl}. Let us quickly recall some definitions and notions needed for the rest of the section.
\smallskip

Given a poset $X$, an \emph{order ideal} $I\subseteq X$ is a set such that $x\in I$ and $y\leq x$ implies $y\in I$. 
The poset $\ideals(X)$ of all order ideals of $X$ (ordered by containment) is a distributive lattice.
An element $\ell$ of a lattice $L$ is \emph{join-irreducible}
if it cannot be written as the join of two other elements, that is, if $\ell=\ell'\vee \ell''$ then $\ell=\ell'$ or $\ell=\ell''$. 
The induced subposet (not sublattice) of $L$ consisting of all join-irreducible elements is denoted by $\joinirrs(L)$

The Fundamental Theorem of Finite Distributive Lattices ({FTFDL})~\cite{Bir-37} states:
\begin{quote}
Up to isomorphism, the finite distributive lattices are exactly the lattices $\ideals(P)$ where $P$ is a finite poset. Moreover, $L$ is isomorphic to $\ideals(\joinirrs(L))$ for every lattice $L$ and $P$ isomorphic to $\joinirrs(\ideals(P))$ for every poset $P$.
\end{quote}

We say that a distributive lattice $L$ is {\em embedded} in $\mathbb{Z}^n$ if 
\begin{itemize}                                                     	

\item the affine hull of $L$ is $\mathbb{Z}^n$ \hfill (full-dimensional),                                                                               \item $L$ is a sublattice of the componentwise order of $\mathbb{Z}^n$ \hfill (sublattice),                                                                               	\item the minimum of $L$ is $\mathbf{0}\in\mathbb{Z}^n$ \hfill (normalized),                                                                               \item if $\ell\prec\ell'$ in $L$ then $\ell-\ell'=e_i$ for some unit vector $e_i$ \hfill (cover-preserving).                                                                               \end{itemize}

Dilworth's Theorem~\cite{Dil-50} generalizes the {FTFDL} to a bijection between embedded distributive sublattices and {{\em chain partitioned posets} (a {\em chain partition} of a poset $X$ is a partition of its ground set into disjoint chains $C_1,\dots ,C_n$)}. More precisely, given a chain partitioned poset $(X,C_1, \ldots, C_n)$ we associate an embedded distributive lattice by mapping ideals $I$ of $X$ to $\mathbb{Z}^n$ via 
$$\phi(I)_i:=|I\cap C_i| \text{  for all }i\in[n].$$ 

Conversely, given a join-irreducible element $\ell$ of the embedded lattice $L$, it covers a unique $\ell'$ and $\ell-\ell'=e_i$ for some unit vector $e_i$. We put $\ell$ into the chain $C_i$. The {FTFDL} corresponds to embeddings into $\{0,1\}^n$ and posets with the trivial singleton-chain partition.
\smallskip

By using Theorem~\ref{thm:polypaths}, the distributive lattice defined in Corollary~\ref{cor:dl} can be viewed in $\mathcal{C}_M$ as follows: for two generalized lattice paths $P$ and $Q$, we have $P\geq Q$ if for every line $l_i$ in the diagram the $y$-coordinate of $P\cap\ell$ is larger or equal than the $y$-coordinate of $Q\cap\ell$. Equivalently, the area below $P$ contains $Q$. 
\smallskip

Let us now describe the distributive lattice structure on $kP_M\cap\mathbb{Z}^{r+m}$ corresponding to elements of $\mathcal{C}_M^k$ (by Corollary~\ref{cor:integ}).
\smallskip

Given the diagram representing $M$ it is easy to construct the chain partitioned poset representing the lattice on these paths, see Figure~\ref{fig:distlattice} for an illustration:
\begin{itemize}
 \item insert lines $T_i$ in the diagram,
 \item in each line $T_i$ add between any two consecutive grid points $k-1$ equidistant points,
 \item connect two of these new points with a line if and only if their difference is $(1,0)$ or $(0,1)$.
 \item remove the points lying on the lower path $L$,
 \item rotate the drawing by $45$ degree clockwise.
\end{itemize}

The resulting diagram is the Hasse diagram of a poset, that we denote by $X^k_M$. The chain partition that we consider simply puts all points on a given line $l_i$ into a chain $C_i$ of $X^k_M$. 

\begin{figure}[ht] 
\centering
 \includegraphics[width=\textwidth]{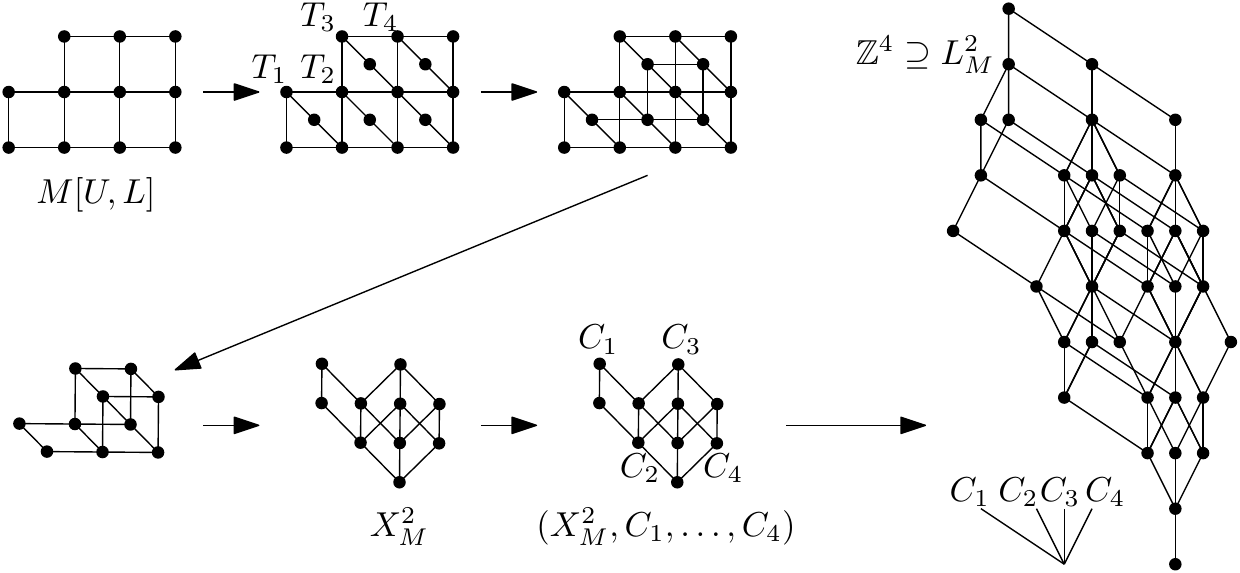}
 \caption{Constructing $X^2_M$, its chain partition, and the embedded $L^2_M\subseteq \mathbb{Z}^4$ from a diagram representing $M$.}
 \label{fig:distlattice}
\end{figure} 

Let $L^k_M$ be the embedded {distributive} lattice obtained from the chain partitioned poset $(X^k_M,C_1, \ldots, C_n)$ via the Dilworth's bijection as described above.

\begin{thm}\label{thm:poset} Let $M$ be a connected rank $r$ LPM on $r+m$ elements.
Then, $L^k_M=kQ_M\cap\mathbb{Z}^{r+m-1}$.
\end{thm}

\begin{proof}
  A generalized lattice path with all its bends on coordinates $(x,y)$ such that $kx,ky,x+y\in\mathbb{N}$ corresponds to choosing the $ky^{th}$ element in chain $C_{x+y}$ of $(X^k_M,C_1, \ldots, C_n)$ for all its bends. Since generalized lattice paths are weakly monotone this choice indeed corresponds to an ideal in $X^k_M$. Conversely, any ideal in $X^k_M$ can be viewed as an element of $\mathcal{C}^k_M$. Thus, the Dilworth mapping satisfies $\phi(X^k_M,C_1, \ldots, C_n)=\mathcal{C}^k_M$. Now, by Corollary~\ref{cor:dl} the ordering on $\mathcal{C}^k_M$ corresponds to the distributive lattice structure on $kQ_M\cap\mathbb{Z}^{r+m-1}$ whose embedding clearly corresponds to the chain partition $C_1, \ldots, C_n$.
\end{proof}

\subsection{Application: Ehrhart polynomial for snakes} 
After having understood the combinatorics of the embedded distributive lattice $L^k_M$ let us combine these results with the polyhedral structure in the case of snakes. 
\smallskip

In this section, let $X$ be a poset on $\{1,\dots ,n\}$ such that this labeling is \emph{natural}, i.e., if $i<_Xj$ then $i<j$. 
The \emph{order polytope} $\mathcal{O}(X)$ of $X$ is defined as the set of those $x\in \mathbb{R}^n$ such that

\begin{equation}\label{o1}
0\le x_i\le 1, \text{ for all } i\in X
\end{equation}

\begin{equation}\label{o2}
x_i\ge x_j, \text{ if } i\le j \text{ in } X
\end{equation}

Note that $\mathcal{O}(X)$ is a convex polytope since it is defined by linear inequalities and is bounded because of \eqref{o1}.
\smallskip

Let $\chi(I)\in\mathbb{R}^n$ denote the characteristic vector of an ideal $I$, i.e.,
$$\chi(I)_i=\left\{ \begin{array}{ll}
1 & \text{ if } i\in I,\\
0 & \text{ if } i\not\in I.\\
\end{array}\right.$$

%
%
%
%
%

It is known~\cite[Corollary 1.3]{Sta-86} that the vertices of $\mathcal{O}(X)$ are the characteristic vectors $\chi(I)$ as $I$ runs through all order ideals $I$ in $X$. In particular, the number of vertices of $\mathcal{O}(X)$ is the number of ideals of $X$.
\smallskip

%

For integers $a_1, \ldots, a_k\geq 2$ denote by $Z(a_1, \ldots ,a_k)$ the \emph{zig-zag-chain} poset on $\sum\limits_{i=1}^ka_i-k+1$ elements arising from $k$ disjoint incomparable chains $C_1, \ldots, C_k$ of lengths $a_1, \ldots a_k$ by identifying the bottom elements of $C_i$ and $C_{i+1}$ for odd $1\leq i<k$ and the top elements of $C_i$ and $C_{i+1}$ for even $1\leq i<k$, see Figure~\ref{fig:chainzigzag}.  

\begin{figure}[ht] 
\centering
 \includegraphics[width=\textwidth]{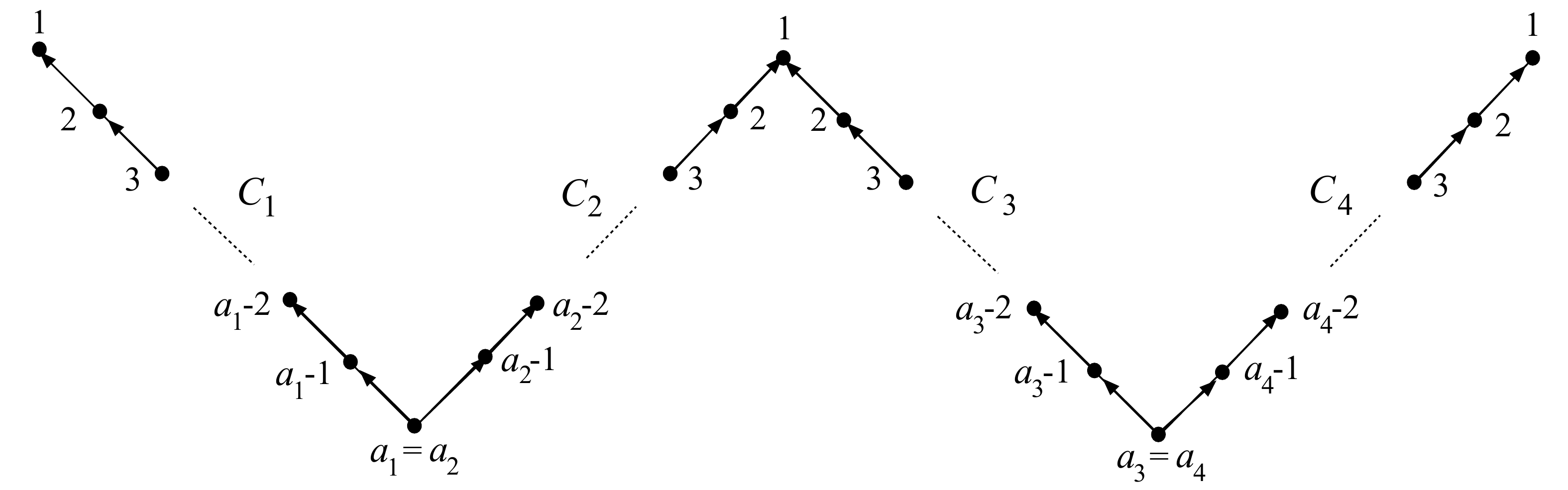}
 \caption{Zig-zag-chain $Z(a_1,a_2,a_3,a_4)$.}\label{fig:chainzigzag}
\end{figure} 

The following result relates snake polytopes with order polytopes.

\begin{thm}\label{thm:orderpoly} Let $a_1, \ldots a_k\geq 2$ be integers. Then
 a connected LPM $M$ is the snake $S(a_1, \ldots a_k)$ if and only if $Q_M$ is the order polytope of the poset $Z(a_1, \ldots ,a_k)$. 
\end{thm}
\begin{proof}
 Let $M$ be a snake of rank $r$ with $r+m$ elements. By Theorem~\ref{thm:dpoly}, we have that $Q_M$ is a full $r+m$-dimensional distributive $(0,1)$-polytope consisting of $q\in\mathbb{R}^{r+m-1}$ such that
 
\begin{equation}\label{eq:orderzig1}
 0\leq (-1)^{L_{i+1}}(q_{i+1}-q_i)\leq 1 \text{ for all } i\in[r+m-2] 
 \end{equation}
 and
\begin{equation}\label{eq:orderzig2}
 0\leq q_i\leq \sum\limits_{j=1}^i(U_j-L_j) \text{ for all } i\in[r+m-1].
 \end{equation}

Now since $L_i=0$ or $1$ for all $i$, we can write \eqref{eq:orderzig1} as

\begin{equation}\label{eq:orderzig3}
q_{i+1}\leq q_i  \text{ if } L_{i+1}=1 \text { and } q_{i}\leq q_{i+1} \text{ if } L_{i+1}=0 \text{ for each } i\in[r+m-2].
 \end{equation}
 
Moreover, since $M$ is a snake then $\sum\limits_{j=1}^i(U_j-L_j)\leq 1$ for all $i\in [r+m-1]$, and thus, from \eqref{eq:orderzig2}, we get that
 
\begin{equation}\label{eq:orderzig4}
0\leq q_i\leq 1  \text{ for all } i\in[r+m-1].
 \end{equation}
 
As $\mathbb{R}^X$ can be identified with $\mathbb{R}^n$ ($X$ on $[n]$) then, by \eqref{o1} and \eqref{o2},  $\mathcal{O}(X)$ is the polytope consisting of those points in $q\in\mathbb{R}^n$ such that 
\begin{equation}\label{eq:polpo}
0\leq q_i\leq 1 \text{ and } q_i\leq q_j \text{ if } i\geq_X j \text{ for all } i,j\in X. 
\end{equation}

Therefore, $\mathcal{O}(Z(a_1,\dots ,a_n))$ is of dimension $\sum\limits_{i=1}^ka_i-k+1$ which is exactly the number of elements of $S(a_1, \ldots a_k)$ and, in this case, it can be verified that inequalities in \eqref{eq:polpo} are given by \eqref{eq:orderzig3} and \eqref{eq:orderzig4}.
\end{proof}

Theorem~\ref{thm:orderpoly} is  useful in order to study Ehrhart polynomials of snake polytopes. Indeed, volumes and Ehrhart polynomials of order polytopes have been already studied. In~\cite[Corollary 4.2]{Sta-86}, it was proved that $$\vol(\mathcal{O}(X))=\frac{e(X)}{n!}$$ where $e(X)$ is the number of linear extensions of $X$ and $n$ the size of $X$. Here a permutation $\sigma$ of $X$ is in the set of linear extensions $\mathcal{L}(X)$ if $i<_Xj$ implies $\sigma(i)<\sigma(j)$.
\smallskip

Given the poset $X$, we define the function $\Omega_X(k)$ as the number of order preserving mappings $\eta$ from $X$ to the total order on $[k]$ (i.e., if $x\le y$ in $X$ then $\eta(x)\le\eta(y)$). In~\cite{Sta-70}, it was proved that $\Omega_X$ is a polynomial (called the \emph{order-polynomial} of $X$). Moreover, it was shown in~\cite[Theorem 2]{Sta-70} that 
 
%

\begin{equation}\label{omega} 
\Omega_X(t+1)=\sum_{s=0}^{n-1}\omega_s\binom{n+t-s}{n}
\end{equation}

\noindent where the sequence $\omega_s$ denotes the number of linear extensions of $X$ such that exactly $s$ consecutive pairs in the linear extension are not ordered as the natural order on $[n]$.
\smallskip

In~\cite[Theorem 4.1]{Sta-86} it was proved that
\begin{equation}\label{omega1} 
\Omega_X(t+1)=L_{\mathcal{O}(X)}(t).
\end{equation}
 
 \begin{rem}\label{rem:oor} In view of \eqref{er-ser} the sequence $\omega$ in \eqref{omega} corresponds to the $h^*$-vector of $\mathcal{O}(X)$. Sometimes $\omega$ is defined in a different (but equivalent) way. For instance, in~\cite{Rei-05} it is defined setting $\mathcal{L}'(X)$ to be the set of permutations $\sigma$ of $X$ such that $i<_Xj$ implies $\sigma^{-1}(i)<\sigma^{-1}(j)$, i.e., the inverse of a linear extension in our sense, and denoting by $\omega_s$ elements $\sigma\in\mathcal{L}'(X)$ such that exactly $s$ consecutive pairs $(i,i+1)$ have $\sigma(i)>\sigma(i+1)$.
\end{rem}


Recall that the {\em rank} of a poset $X$ is the length of its largest chain and that $X$ is {\em graded} if all maximal chains have the same length.  A vector $(c_0,\dots ,c_d)$ is {\em unimodal} if there exists an index $p, 0\le p\le d$, such that $c_{i-1}\le c_i$ for $i\le p$ and $c_{j}\ge c_{j+1}$ for $j\ge p$.
\smallskip

\begin{thm}\label{thm:uni1} Let $a,b\ge 2$ be integers.
The $h^*$-vectors of the snake polytopes $P_{S(a,\ldots, a)}$ and $P_{S(a,b)}$ are unimodal.
\end{thm} 

\begin{proof} 
By Theorem~\ref{thm:dpoly} we have that $P_{S(a,\ldots, a)}$ and $P_{S(a,b)}$ have the same Ehrhart polynomial as $Q_{S(a,\ldots, a)}$ and $Q_{S(a,b)}$. By Theorem~\ref{thm:orderpoly} both these have the same Ehrhart polynomial as $\mathcal{O}(Z(a,\ldots, a))$ and $\mathcal{O}(Z(a,b))$. Now, Remark \ref{rem:oor} the $h^*$-vectors of the latter two polytopes coincide with the sequence $\omega$ in \eqref{omega} associated to $Z(a,\ldots, a)$ and $Z(a,b)$, respectively. In~\cite{Rei-05}, it was proved that if $X$ is a graded poset then the sequence $\omega$ is unimodal. Since the zig-zag-chain poset $Z(a,\ldots, a)$ is graded, the latter implies that the $h^*$-vector of $P_{S(a,\dots , a)}$ is unimodal. 
\smallskip

Now, observe that all linear extensions of $Z(a,b)$ begin with the unique minimal element, which thus can be removed from $Z(a,b)$ without affecting $\omega$. The resulting poset is a disjoint union of two chains. It can be easily shown by hand, that $\omega$ is unimodal in this case, but for shortness let us just refer to the stronger result of~\cite{Sim-84} showing unimodality of $\omega$ for any disjoint union of chains. Thus, implying the unimodularity of the $h^*$-vector of $P_{S(a,b)}$.
%
%
%
%
%
%
\end{proof}

\section{Concluding remarks}\label{sec:concluding}

{Theorem~\ref{thm:uni1} supports the more general conjecture due to Loera, Haws, and K\"oppe~\cite[Conjecture 2]{Loer-09} asserting that the $h^*$-vector of any matroid basis polytope is unimodal.}

{A famous conjecture due to Neggers~\cite{Neg-78} states that the real-rootedness of the polynomial with coefficients $h^*$ associated to an order polytope implies unimodality of $h^*$. Since Neggers' conjecture holds for all naturally labeled posets on at most $8$ elements (see~\cite{Stem-97}) then, by Theorem \ref{thm:orderpoly}, we obtain the unimodality of $h^*$ for snakes $S(a_1,\ldots, a_k)$ with $a_1+\ldots+a_k-(k-1)\leq 8$.
Moreover, it follows from a result in~\cite{Wag-92} that if the $h^*$-polynomials of $P$ and $Q$ are real-rooted then so is the $h^*$-polynomial of their product $P\times Q$. Since the result of~\cite{Sim-84} used for $S(a,b)$ in Theorem~\ref{thm:uni1} indeed states real-rootedness, we get that Theorem~\ref{thm:uni1} extends to LPMs that are direct sums of snakes of the form $S(a,b)$. However, for graded posets real-rootedness is not known, see~\cite[Question 2]{Bra-15}, and thus unimodality of $h^*$ for direct sums of snakes of the form $S(a, \ldots, a)$
cannot be concluded.}

{Finally, while Neggers' conjecture has been disproved for general posets~\cite{Stem-07}, unimodality of $h^*$ for order polytopes remains open, see~\cite[Question 1]{Bra-15}.} { In view of our results, perhaps zig-zag-chain posets are good candidates to be investigated with respect to these questions.}
%
%
%
%
%
%

\bibliography{lpbib}
\bibliographystyle{my-siam}

\end{document}